\documentclass[12pt]{article}
\usepackage{ae}
\usepackage[T1]{fontenc}
\usepackage[nottoc]{tocbibind}
\usepackage{amsmath}
\usepackage{amsthm}
\usepackage{amssymb}
\usepackage{braket} 
\usepackage{mathtools} 
\usepackage{dsfont} 
\usepackage{accents} 
\usepackage{upgreek}

\usepackage{color}

\newcounter{commento}
\numberwithin{equation}{section} 

\makeatletter
\let\origthm@notefont\thm@notefont
\renewcommand{\thm@notefont}[1]{\origthm@notefont{#1\the\thm@headfont}}
\makeatother

\newtheorem{thm}{Theorem}[section]
\newtheorem{theorem}[thm]{Theorem}
\newtheorem{corollary}[thm]{Corollary}
\newtheorem{lemma}[thm]{Lemma}
\newtheorem{proposition}[thm]{Proposition}

\theoremstyle{definition}
\newtheorem{definition}[thm]{Definition}
\newtheorem{example}[thm]{Example}
\newtheorem{observation}[thm]{Observation}
\newtheorem{remark}[thm]{Remark}
\newtheorem{notation}[thm]{Notation}

\newcommand{\cat}[1]{{\boldsymbol{\mathsf{#1}}}} 
\newcommand{\Sets}{\cat{Set}}

\newcommand{\SMan}{\cat{SMan}}
\newcommand{\AoMan}{\cat{A_0 Man}}
\newcommand{\cAoMan}{\cat{\mathcal{A}_0 Man}}

\newcommand{\SAlg}{\cat{SAlg}}
\newcommand{\Gras}{\cat{\Lambda}}
\newcommand{\SWA}{\cat{SWA}}

\newcommand{\op}{^{\mathrm{op}}} 
\newcommand{\funct}[2]{\left[{#1},{#2}\right]}
\newcommand{\cfunct}[2]{\left[{#1}\op,{#2}\right]}
\newcommand{\dfunct}[2]{\left[\left[{#1},{#2}\right]\right]}
\newcommand{\subfunc}[2]{#1_{#2}}

\newcommand{\FOP}[1]{#1(\blank)} 
\newcommand{\fop}[2]{#1(#2)} 

\newcommand{\yon}{\mathcal{Y}} 
\newcommand{\ber}{\mathcal{B}} 
\newcommand{\sch}{\mathcal{S}} 

\newcommand{\nset}[1]{{\mathds{#1}}} 
\newcommand{\N}{\nset{N}}
\newcommand{\Z}{\nset{Z}}
\newcommand{\R}{\nset{R}}

\renewcommand{\to}{\mathchoice{\longrightarrow}{\rightarrow}{\rightarrow}
{\rightarrow}}
\let\mto\mapsto \renewcommand{\mapsto}{\mathchoice{\longmapsto}
{\mto}{\mto}{\mto}}

\newcommand{\id}{\mathds{1}} 
\newcommand{\pr}{\mathrm{pr}} 
\newcommand{\isom}{\cong} 

\newcommand{\Hom}{\mathrm{Hom}}
\newcommand{\HOM}{\underline{\Hom}}
\newcommand{\Der}{\mathrm{Der}}

\newcommand{\restr}[2]{{#1}_{|{#2}}} 

\newcommand{\sheaf}[1][O]{\mathcal{#1}} 
\newcommand{\stalk}[3][O]{\sheaf[#1]_{#2,#3}} 
\newcommand{\distr}[3][]{\sheaf_{#2,#3}^{#1*}} 
\newcommand{\red}[1]{\widetilde #1} 

\newcommand{\topo}[1]{\abs{#1}} 
\newcommand{\maxid}{\mathcal{M}} 
\newcommand{\ev}{\mathrm{ev}} 

\newcommand{\Cinf}{\sheaf[C]^\infty} 
\newcommand{\functions}{\sheaf[F]} 

\newcommand{\p}[1]{p(#1)} 

\newcommand{\nil}[1]{\accentset{\circ}#1} 

\newcommand{\ext}[2][]{\Lambda_{#1}(#2)} 
\newcommand{\extn}[1]{\Lambda_{#1}} 

\newcommand{\fseries}[2]{\mathfrak{A}_{#1|#2}} 
\newcommand{\rspoly}[2]{\R[#1|#2]} 

\newcommand{\funcpt}[1]{\underline #1} 

\newcommand{\X}{\mathrm{x}}
\newcommand{\Y}{\mathrm{y}}
\newcommand{\T}{\uptheta}

\newcommand{\nbd}{\protect\nobreakdash-\hspace{0pt}} 

\DeclarePairedDelimiter{\abs}{\lvert}{\rvert} 
\newcommand{\pair}[2]{\big\langle #1, #2 \big\rangle} 

\newcommand{\blank}{\mspace{2mu}{\cdot}\mspace{2mu}} 
\newcommand{\pd}[2]{\frac{\partial #1}{\partial #2}} 
\newcommand{\supp}{\mathrm{supp}} 

\renewcommand{\phi}{\varphi}
\renewcommand{\theta}{\vartheta}
\renewcommand{\epsilon}{\varepsilon}

\newcommand{\cA}{\mathcal{A}}
\newcommand{\cF}{\mathcal{F}}
\newcommand{\cG}{\mathcal{G}}
\newcommand{\cU}{\mathcal{U}}
\newcommand{\F}{\boldsymbol{F}}
\newcommand{\f}{\boldsymbol{f}}
\newcommand{\bU}{\accentset{\frown}{U}}
\newcommand{\bV}{\accentset{\frown}{V}}

\begin{document}

\centerline{\Large\bf The local functors of points of
supermanifolds}

\bigskip

\centerline{L. Balduzzi$^\natural$, C. Carmeli$^\natural$, R.
Fioresi$^\flat$}

\medskip

\centerline{\footnotesize e-mail: luigi.balduzzi@{ge.infn.it},
claudio.carmeli@{ge.infn.it}, fioresi@dm.unibo.it}

\medskip


\begin{abstract}
We study the local functor of points (which we call the
Weil--Berezin functor) for
smooth supermanifolds, providing a characterization,
representability theorems and
applications to differential calculus.
\end{abstract}


\section{Introduction}

Since the 1970s the foundations of supergeometry have been
investigated by several physicists and mathematicians. Most of the
treatments (e.~g.\ \cite{BL, Kostant, Berezin, Leites, Manin, DM,
Varadarajan}) present supermanifolds as classical manifolds where
the structure sheaf is  modified so that the sections are allowed to
take values in
$\Z_2$\nbd graded commutative algebras 
and the sheaf itself is assumed to be locally of the form
$\Cinf(\R^p)\otimes \extn{q}$, with $\extn{q}$ denoting the
Grassmann algebra in $q$ generators. This approach is very much in
the spirit of classical algebraic geometry and dates back to the
seminal works of F. A. Berezin and D. A. Le{\u\i}tes \cite{BL}, and
B. Kostant \cite{Kostant}.

\medskip

It is nevertheless only later in \cite{Manin, DM}, that the
parallelism with classical algebraic geometry is fully worked out
and the functorial language starts to be used systematically. In
particular the functor of points approach becomes a {powerful}
device allowing, among other things, one to recover some geometric
intuition by giving a rigorous meaning to otherwise just formal
expressions.
In this approach, a supermanifold $M$ is fully recovered by the
knowledge of its functor of points, $S \mapsto \fop{M}{S} \coloneqq
\Hom(S,M)$, which associates to a supermanifold $M$, the set of its
$S$\nbd points for every supermanifold $S$. The crucial result in
this context is Yoneda's lemma which establishes a bijective
correspondence between morphisms of supermanifolds and natural
transformations between their corresponding functors of points.

\medskip


Other approaches to the theory of supermanifolds involving new local
models and possibly non Hausdorff topologies were developed later
(\cite{Batchelor, Rogers80, DeWitt, Schwarz}).
For a detailed review of some of these approaches, that we do not
pursue here, we refer the reader to \cite{BBHR, Rogers07}.

\medskip

This paper is devoted to understanding the approach to supermanifolds
theory via the \emph{local functor of points}, which associates to
each {smooth} supermanifold $M$ the set of its $A$\nbd points for
all super Weil algebras $A$. These are finite dimensional
commutative superalgebras of the form $A = \R \oplus \nil{A}$ with
$\nil{A}$ a nilpotent ideal. The set of the $A$\nbd points of the
smooth supermanifold $M$ is defined as $M_A=\Hom_{\SAlg}(\sheaf(M),
A)$, in striking analogy with the functor of points previously
described. In fact, when $A$ is a {finite dimensional} Grassmann
algebra, $M_A$ is indeed the set of the $\R^{0|q}$\nbd points of the
supermanifold $M$ in the sense specified above, for suitable $q$. As
we have defined it, the local functor of points does \emph{not}
determine the supermanifold,
unless we put an extra structure on $M_A$, in
other words, unless we carefully define the category image for
the functor $A \mapsto M_A$.

\medskip

Our approach is a slight modification of the one in \cite{Schwarz,
Voronov}, by Schwarz and Voronov, the main difference being that
they consider Grassmann algebras instead of super Weil algebras. In
this sense our work is mainly providing additional insight into well
known results and clarifies the representability issues
often overlooked in most of the literature.
Moreover the local functor of points that we
examine in our work (\emph{Weil--Berezin functor}) has the advantage
of being able to bring differential calculus naturally into the picture.%

\medskip

The paper is organized as follows.

\medskip

In section \ref{sec:basic_def} we review some  basic definitions of
supergeometry like the definition of superspace, supermanifold and
its associated functor of points.

\medskip

In section \ref{sec:SWA_A-points} we introduce super Weil algebras
with their basic properties
and we define the functor of the $A$-points of a supermanifold $M$,
$A \mapsto M_A$ from the category of super Weil algebras to the
category of sets. 
We show
this functor does not characterize the
supermanifold $M$. In order to obtain this,
the image category needs to be suitably specialized by giving to
each set $M_A$ an extra structure.

\medskip

In section \ref{sec:WB_functor-Sh_emb},
we obtain a bijective correspondence between supermanifold morphisms
and natural transformations between the functors of $A$-points, by
endowing
the set $M_A$ with the structure of an $A_0$\nbd smooth manifold.
For this new functor, called the \emph{Weil--Berezin functor of $M$}
the analogue of
Yoneda's lemma holds 
and, as a consequence, supermanifolds embed in a full and faithful way
into the category of Weil--Berezin functors (\emph{Schwarz
embedding})
and we can prove a representability theorem.
We end the section by giving  a brief account of  the functor of
$\Lambda$\nbd points originally described by Schwarz, which is the
restriction of the Weil--Berezin functor to Grassmann algebras.

\medskip

In section \ref{sec:diff_calc} we examine some aspects of super
differential calculus on supermanifolds in the language of the
Weil--Berezin functor,
establishing a connection between our treatment and Kostant's
seminal approach to supergeometry and proving the Weil transitivity
theorem.


\paragraph{Acknowledgements.} We want to thank prof.\ G. Cassinelli, prof.\ M.
Duflo, prof. P. Michor
and prof.\ V. S. Varadarajan for helpful
discussions.

\section{Basic Definitions of Supergeometry} \label{sec:basic_def}

In this section we recall few basic definitions in supergeometry.
Our main references are \cite{Kostant, Manin, DM, Varadarajan}.


\medskip

Let $\R$ be our ground field.

\medskip

A \emph{super vector space} is a $\Z_2$\nbd graded vector space,
i.~e.\ $V=V_0 \oplus V_1$; the elements in $V_0$ are called
\emph{even}, those in $V_1$ \emph{odd}. An element $v \neq 0$ in
$V_0 \cup V_1$ is said \emph{homogeneous} and $\p{v}$ denotes its
parity: $\p{v}= 0$ if $v \in V_0$, $\p{v}=1$ if $v \in V_1$.
$\R^{p|q}$ denotes the super vector space $\R^p \oplus \R^q$. A
\emph{superalgebra} $A$ is an algebra that is also a super vector
space, $A=A_0 \oplus A_1$, and such that $A_i A_j \subseteq
A_{i+j\pmod 2}$. $A_0$ is an algebra, while $A_1$ is an $A_0$\nbd
module. $A$ is said to be \emph{commutative} if for any two
homogeneous elements $x$ and $y$,
    $xy = (-1)^{\p{x}\p{y}} yx$. 
The category of real commutative superalgebras is denoted by
$\SAlg$ and all our superalgebras are assumed to be in $\SAlg$.

\begin{definition}
A \emph{superspace} $S=(\topo{S}, \sheaf_S)$ is a topological space
$\topo{S}$, endowed with a sheaf of superalgebras $\sheaf_S$ such
that the stalk at each point  $x \in \topo{S}$, denoted by
$\sheaf_{S,x}$, is a local superalgebra (i.e. it has a unique
graded maximal ideal).
A \emph{morphism} $\phi
\colon S \to T$ of superspaces is a pair $(\topo{\phi}, \phi^*)$,
where $\topo{\phi} \colon \topo{S} \to \topo{T}$ is a continuous map
of topological spaces and $\phi^* \colon \sheaf_T \to \topo{\phi}_*
\sheaf_S$, called \emph{pullback}, is such that
$\phi_x^*(\maxid_{\topo{\phi}(x)}) \subseteq \maxid_x$ where
$\maxid_{\topo{\phi}(x)}$ and $\maxid_{x}$ denote the maximal ideals
in the stalks $\sheaf_{T,\topo{\phi}(x)}$ and $\sheaf_{S,x}$
respectively.
\end{definition}

\begin{example}[The smooth local model]
The superspace $\R^{p|q}$ is the topological space $\R^p$ endowed
with the following sheaf of superalgebras. For any open set $U
\subseteq \R^p$ define
$\sheaf_{\R^{p|q}}(U) \coloneqq \Cinf_{\R^p}(U) \otimes \ext{\theta_1,\dots,\theta_q}$,
where $\ext{\theta_1,\dots,\theta_q}$ is the real exterior algebra
(or \emph{Grassmann algebra}) generated by the $q$ variables
$\theta_1,\dots,\theta_q$ and $\Cinf_{\R^p}$ denotes the $\Cinf$
sheaf on $\R^p$.
\end{example}


\begin{definition} \label{def:supermanifold}
A (smooth) \emph{supermanifold} of dimension $p|q$ is a superspace
$M=(\topo{M},\sheaf_M)$ which is locally isomorphic to $\R^{p|q}$,
i.~e.\ for all $x \in \topo{M}$ there exist  open sets $x \in V_x
\subseteq \topo{M}$ and $U \subseteq \R^{p}$ such that:
$    \restr{\sheaf_{M}}{V_x} \isom \restr{\sheaf_{\R^{p|q}}}{U}$.
In particular supermanifolds of the form
$(U,\restr{\sheaf_{\R^{p|q}}}{U})$ are called \emph{superdomains}. A
\emph{morphism} of supermanifolds is simply a morphism of
superspaces. $\SMan$ denotes the category of supermanifolds.
We shall denote with $\sheaf(M)$ the superalgebra $\sheaf_M(\topo{M})$ of
global sections on the supermanifold $M$.
\end{definition}

If $U$ is open in $\topo{M}$, $(U,\restr{\sheaf_M}{U})$ is also a
supermanifold and it is called the \emph{open supermanifold
associated with $U$}. We shall often refer to it just by $U$,
whenever no confusion is possible.


\medskip

Suppose $M$ is a supermanifold and $U$ is an open subset of
$\topo{M}$. Let $\sheaf[J]_M(U)$ be the ideal of the nilpotent
elements of $\sheaf_M(U)$. $\sheaf_M/\sheaf[J]_M$ defines a sheaf of
purely even algebras over $\topo{M}$ locally isomorphic to
$\Cinf(\R^p)$. Therefore $\red{M} \coloneqq
(\topo{M},\sheaf_M/\sheaf[J]_M)$ defines a classical smooth
manifold, called the \emph{reduced manifold} associated with $M$.
The projection $s \mapsto \red{s} \coloneqq s + \sheaf[J]_M(U)$,
with $s \in \sheaf_M(U)$, is the pullback of the embedding $\red{M}
\to M$.
If $\phi$ is a supermanifold morphism, since
$\topo{\phi}^*(\red{s}) = \red{{\phi^*(s)}}$, 
the morphism $\topo{\phi}$ is automatically
smooth.
%

\medskip

There are several equivalent ways to assign a morphism between two
supermanifolds. The following result can be found in
\cite[ch.~4]{Manin}.

\begin{theorem}[Chart theorem] \label{th:morphisms}
Let $U$ and $V$ two smooth superdomains, i.~e.\ two open
subsupermanifolds of $\R^{p|q}$ and
$\R^{m|n}$ respectively. There is a bijective correspondence between \\
1. superspace morphisms $U \to V$; \\
2. superalgebra morphisms $\sheaf(V) \to \sheaf(U)$; \\
3. the set of pullbacks of a fixed coordinate system on
    $V$, i.~e.\ $(m|n)$\nbd uples
    \[
        (s_1,\ldots,s_m,t_1,\ldots,t_n) \in
        \sheaf(U)_0^m \times \sheaf(U)_1^n
    \]
    such that $\big(\red{s_1}(x),\ldots,\red{s_m}(x)\big)
    \in \topo{V}$ for each $x \in \topo{U}$.
\end{theorem}

Any supermanifold morphism $M \to N$ is then uniquely determined by
a collection of local maps, once  atlases on $M$ and $N$ have been
fixed.  A morphism can hence be given by describing it in local
coordinates.

Since we are considering the smooth category a further
simplification occurs: we can assign a morphism between
supermanifolds by assigning the pullbacks of the global sections
(see \cite[\S~2.15]{Kostant}), i.~e.
\begin{equation} \label{eq:pullback_of_global_sections}
    \Hom_{\SMan}(M,N) \isom \Hom_{\SAlg} \big( \sheaf(N) ,
\sheaf(M) \big) \text{.}
\end{equation}


\medskip

The theory of supermanifolds resembles very closely the classical
theory. One can, for example, define tangent bundles, vector fields
and the differential of a morphism similarly to the classical case.
For more details see \cite{Kostant, Leites, Manin, DM, Varadarajan}.

\medskip

Due to the presence of nilpotent elements in the structure sheaf of
a supermanifold, supergeometry can also be equivalently and very
effectively studied using the language of \emph{functor of points},
a very useful tool in algebraic geometry.

\medskip

Let us first fix some notation we will use throughout the paper. If
$\cat{A}$ and $\cat{B}$ are two categories,
$\funct{\cat{A}}{\cat{B}}$  denotes the category of functors between
$\cat{A}$ and $\cat{B}$ (notice that in general
$\funct{\cat{A}}{\cat{B}}$ won't have small hom-sets). Clearly, the
morphisms in $\funct{\cat{A}}{\cat{B}}$ are the natural
transformations. Moreover we denote by $\cat{A}\op$ the
\emph{opposite category} of $\cat{A}$, so that the category of
contravariant functors between $\cat{A}$ and $\cat{B}$ is identified
with $\cfunct{\cat{A}}{\cat{B}}$
(see \cite{MacLane}).

\begin{definition}
Given a supermanifold $M$, we define its \emph{functor of
points}
\[
    \FOP{M} \colon \SMan\op \to \Sets, \qquad S \mapsto \fop{M}{S} \coloneqq
\Hom(S,M)
\]
as the functor from the opposite category of supermanifolds to the
category of sets defined on the morphisms as usual: $M(\phi)f=f
\circ \phi$,
where $\phi \colon T \to S$, $f \in M(S)$.
The elements in $\fop{M}{S}$ are also
called the \emph{$S$\nbd points} of $M$.
\end{definition}

Given two supermanifolds $M$ and $N$, Yoneda's lemma (a general
result valid for all categories with small hom-sets) establishes
a bijective correspondence
\begin{eqnarray*}
    \Hom_{\SMan}(M,N) & \longleftrightarrow &
    \Hom_{\cfunct{\SMan}{\Sets}} \big( \FOP{M}, \FOP{N} \big)
\end{eqnarray*}
between the morphisms $M \to N$ and the natural transformations
$\FOP{M} \to \FOP{N}$ (see \cite[ch.~3]{MacLane} or
\cite[ch.~6]{EH}). This allows us to view a morphism of
supermanifolds as a family of morphisms $\fop{M}{S} \to \fop{N}{S}$
depending functorially on the supermanifold $S$. In other words,
Yoneda's lemma provides a full and faithful immersion
\[
    \yon \colon \SMan \to \cfunct{\SMan}{\Sets}.
\]
There are however objects in $\cfunct{\SMan}{\Sets}$ that do not
arise as the functors of points of a supermanifold. We say that a
functor $\cF \in \cfunct{\SMan}{\Sets}$ is \emph{representable} if
it is isomorphic to the functor of points of a supermanifold.



We now want to recall a representability criterion, which allows to
single out, among all the functors from the category of
supermanifolds to sets, those that are representable,
(see \cite[ch.~1]{DG}, \cite[A.13]{FLV} for more details).

\begin{theorem}[Representability criterion] \label{theor:representability}
A functor $\cF \colon \SMan\op \to \Sets$ is representable if and
only if: \\
1. $\cF$ is a sheaf, i.~e.\ it has the sheaf property; \\
2. $\cF$ is covered by open supermanifold subfunctors $\set{\cU_i}$.
\end{theorem}

\section{Super Weil algebras and $A$\nbd points} \label{sec:SWA_A-points}

In this section we introduce the category $\SWA$ of super Weil
algebras. These are finite dimensional commutative superalgebras
with a nilpotent graded ideal of codimension one.  Super Weil
algebras are the basic ingredient in the definition of the
Weil--Berezin functor and the Schwarz embedding. The simplest
examples of super Weil algebras are {finite dimensional} Grassmann
algebras. These are the only super Weil algebras that can be
interpreted as algebras of global sections of supermanifolds, namely
$\R^{0|q}$.



\medskip

We now define the category of \emph{super Weil algebras}. The
treatment follows closely that contained in \cite[\S~35]{KMS} for
the classical case.

\begin{definition}
We say that $A$ is a (real) \emph{super Weil algebra} if
it is a commutative unital superalgebra over $\R$ and \\
1. $\dim A < \infty$, \\
2. $A=\R \oplus \nil{A}$,
where $\nil{A}=\nil{A}_0 \oplus \nil{A}_1$ is a graded nilpotent ideal. \\
The category of super Weil algebras is denoted by $\SWA$.
The \emph{height} of $A$ is the lowest $r$ such that
$\nil{A}^{r+1}=0$ and the \emph{width} of $A$ is the dimension of
$\nil{A}/\nil{A}^2$.
Notice  that super Weil algebras are \emph{local superalgebras}, i.~e.\
they contain a unique maximal graded ideal.
\end{definition}

\begin{remark}
As a direct consequence of the definition, each super Weil algebra
has an associated short exact sequence:
\[
    0 \to \R \stackrel{j_A}{\to} A = \R \oplus \nil{A} \stackrel{\pr_A}{\to} A/\nil{A} \isom \R \to 0 \text{.}
\]
Clearly the sequence splits and each $a \in A$ can be written
uniquely as
$    a = \red{a} + \nil{a}$
with $\red{a} \in \R$ and $\nil{a} \in \nil{A}$.
\end{remark}

\begin{example}[Dual numbers and super dual numbers] \label{example:SDN}
The simplest
example of super Weil algebra in the classical setting is
$\R(x)=\R[x]/\langle x^2 \rangle$ the algebra of dual numbers. Here
$x$ is an even indeterminate. 
Similarly we have the super dual numbers: $\R(x,\theta) =
\R[x,\theta] / \langle x^2,x\theta,\theta^2 \rangle$ where $x$ and
$\theta$ are respectively even and odd indeterminates.
\end{example}

\begin{example}[Grassmann algebras]
The polynomial algebra in $q$ odd variables $\ext{\theta_1,\dots,
\theta_q}$ is another example of super Weil algebra. {Finite
dimensional} Grassmann algebras are actually a full subcategory of
$\SWA$.
\end{example}

\begin{lemma}\label{lemma:SWA}
Let    $\rspoly{k}{l} \coloneqq \R[x_1,\ldots,x_k]
\otimes \ext{\theta_1,\ldots,\theta_l}$
denote the superalgebra of real polynomials
in $k$ even and $l$ odd variables.
The following are equivalent: \\
1. $A$ is a super Weil algebra;\\
2. $A \cong \sheaf_{\R^{p|q},0}/J$ for suitable $p,q$ and
$J$ graded ideal
containing a power of the maximal ideal $\maxid_0$ in the stalk
$\sheaf_{\R^{p|q},0}$;\\
3. $A \cong \rspoly{k}{l} /I$ for a suitable graded ideal $I
\supseteq
\langle x_1,\ldots,x_k, \theta_1, \ldots, \theta_l \rangle^k$.
\end{lemma}

\begin{proof}
We leave this to the reader as an exercise.
\end{proof}

\begin{definition}
Let $M$ be a supermanifold and $A$ a super Weil algebra. We
define the \emph{set of $A$\nbd points} of $M$,
\[
    M_{A} \coloneqq \Hom_\SAlg (\sheaf(M), A)
\]
\end{definition}


We can define the functor $M_{(\blank)} : \SWA \to \Sets$,
on the objects as $A \mapsto M_A$
and on morphisms as $\rho \mapsto \funcpt{\rho}$ with $\rho \in
\Hom_{\SAlg}(A,B)$ and
$\funcpt{\rho} \colon x_A \mapsto \rho
\circ x_A$.

\begin{remark} \label{remark:localalg}
Observe that the only super Weil algebras which are equal to
$\sheaf(M)$ for some supermanifold $M$ are those of the form
$\ext{\theta_1,\dots,\theta_q} = \sheaf(\R^{0|q})$. In fact as
soon as $M$ has a nontrivial even part, the algebra $\sheaf(M)$
becomes infinite dimensional. For this reason this functor is quite
different from the functor of points introduced previously.
\end{remark}



Let us recall a well known classical result.

\begin{lemma}[``Super" Milnor's exercise]
Denote by $M$  a smooth supermanifold. The superalgebra maps
$\sheaf(M) \to \R$ are exactly the evaluations $\ev_x \colon s
\mapsto \red{s}(x)$ in the points $x \in \topo{M}$. In other words
there is a bijective correspondence between $M_\R =
\Hom_{\SAlg}\big(\sheaf(M),\R\big)$
and $\topo{M}$.
\end{lemma}

\begin{proof}
This is a simple consequence of the chart theorem \ref{th:morphisms}
and eq.\ \eqref{eq:pullback_of_global_sections}, considering that
$\sheaf(\R^{0|0}) = \R$ and the pullback of a morphism $\phi \colon
\R^{0|0} \to M$ is the evaluation at $\topo{\phi}(\R^0)$.
\end{proof}

Let $x_A \in M_A$. Due to the previous lemma, there exists a unique
point of $\topo{M}$, that we denote by $\red{x_A}$, such that
$\pr_{A} \circ x_A = \ev_{\red{x_A}}$, where $\pr_A$ is the
projection $A \to \R$. We thus have a map
\begin{equation} \label{eq:base_point}
    \begin{aligned}
        \Hom_\SAlg \big( \sheaf(M),A \big) &\to \Hom_\SAlg \big( \sheaf(M),\R \big) \isom \topo{M} \\
        x_A &\mapsto \pr_{A} \circ x_A = \ev_{\red{x_A}} \text{.}
    \end{aligned}
\end{equation}
We say that $\red{x_A}$ is the \emph{base
point} of $x_A$ or that $x_A$ is an $A$\nbd point \emph{near} $\red{x_A}$.
We denote with $M_{A,x}$ the set of $A$-points near $x\in\topo{M}$.

The next proposition asserts the local nature of the functor of the
$A$-points.

\begin{proposition}
Let $M$ be a smooth supermanifold. Let $s \in \sheaf(M)$ and let
$x_A\in \Hom_\SAlg\big (\sheaf(M), A \big)$. Assume that $s$ is
zero when restricted to a certain neighbourhood of $\red{x_A}$ (see
eq.\ \eqref{eq:base_point}). Then $x_A(s)=0$.
\end{proposition}

\begin{proof}
Suppose $U \ni \red{x_A}$ is such that $\restr{s}{U}=0$. Let $t \in
\sheaf_M(U)$ be such that $\supp(t) \subset U$ and $\restr{t}{V} =
1$, where the closure of $V$ is contained in $U$. Then
  $  0 = x_A(st) = x_A(s)x_A(t)$.
So $x_A(s)=0$, since $x_A(t)$ is invertible because of
$\ev_{\red{x_A}}(t) = 1$, where $\ev_{\red{x_A}}$ denotes the
evaluation at $\red{x_A}$.
\end{proof}



\begin{observation} \label{obs:stalk}
The above proposition shows that $x_A(s)$ depends only on the germ
of $s$ in $\red{x_A}$, i.~e.\ $x_A$ is also a superalgebra map from the
stalk $\stalk{M}{\red{x_A}}$ of $\sheaf_M$ in $\red{x_A}$ to $A$.
Therefore it is possible to give a meaning to $x_A([s])$
for a germ $[s]$ in $\stalk{M}{\red{x_A}}$.
It is not hard to show that
$M_A \isom \bigsqcup_{x \in \topo{M}}
\Hom_\SAlg(\stalk{M}{x},A)$. This identification
allows to extend the definition
of the local functor of points to the category of holomorphic or
real analytic
supermanifolds. Many of the results we prove extend relatively
easily to the holomorphic (or real analytic)
category, but we shall not pursue this
point of view in the present paper.
\end{observation}

\begin{notation}
Here we introduce a multi-index notation that we will use in the
following. Let $\set{x_1,\ldots,x_p,\theta_1,\ldots,\theta_q}$ be a
system of coordinates. If $\nu = (\nu_1,\ldots,\nu_p) \in \N^p$, $J
= \set{j_1,\ldots,j_r} \subseteq \set{1,\ldots,q}$, with $1 \leq j_1
< \dots < j_r \leq q$, we define
$    x^\nu \coloneqq x_1^{\nu_1} x_2^{\nu_2} \cdots x_p^{\mu_p} \text{,}$
$\theta^J \coloneqq \theta_{j_1} \theta_{j_2} \cdots
    \theta_{j_r}$.
Moreover  we set $\nu! \coloneqq \prod_i \nu_i!$,
 $\abs{\nu} \coloneqq \sum_i
\nu_i$ and $\abs{J}$ the cardinality of $J$.
\end{notation}

In order to obtain further information about the structure of $M_A$
we need some preparation. Next lemma gives some insight on the
structure of the stalk at a given point (for the proof see
\cite[\S~2.1.8]{Leites} or \cite[ch.~4]{Varadarajan}).

\begin{lemma}[Hadamard's lemma] \label{lemma:polynomials}
Let $M$ be supermanifold, $x \in \topo{M}$ and
$\set{x_i,\theta_j}$ is a system of coordinates in a neighborhood
$U$ of $x$. Denote by $\maxid_{U,x}$ the ideal of the sections in
$\sheaf_M(U)$ whose value at $x$ is zero.
For each $s \in \sheaf_M(U)$ and
$k \in \N$ there exists a polynomial $P$ in ${x_i}$ and ${\theta_j}$
such that $s-P \in \maxid_{U,x}^k$.
\end{lemma}


As a consequence we have the following proposition.

\begin{proposition}
\label{prop:valori_assegnati}
Each element $x_A$ of $M_A$ is
    determined by the images of
    a system of local coordinates around
    $\red{x_A}$. Conversely, given $x \in \topo{M}$, a system of
    local coordinates $\set{x_i}_{i=1}^p$,
    $\set{\theta_j}_{j=1}^q$ around $x$, and elements
    $\set{\X_i}_{i=1}^p$, $\set{\T_j}_{j=1}^q$, $\X_i \in
    A_0$, $\T_j \in A_1$,\footnote{%
        The reader should notice the difference between
        $\set{x_i,\theta_j}$ and $\set{\X_i,\T_j}$.
    } such that $\red{\X_i} = \red{x_i}(x)$, there
    exists a unique morphism $x_A \in \Hom_\SAlg(\sheaf(M),A)$ with
    $x_A({x_i}) = \X_i$, $x_A({\theta_j}) = \T_j$.

\end{proposition}

\begin{proof}
Suppose that $x_A$ is given. We want to show that $x_A({x_i})$,
$x_A({\theta_j})$
determine $x_A$ completely. This follows noticing that \\
1. the image of a polynomial section under $x_A$ is determined, \\
2. there exists $k \in \N$ such that the kernel of $x_A$ contains
    $\maxid_{U,x}^k$ (see \ref{lemma:SWA}),
and using previous lemma. We now come to existence.
Suppose that the
images of the coordinates are fixed as in the hypothesis
and let $s$ in $\sheaf_M(U)$. We define $x_A({s})$
through a formal Taylor expansion. More precisely let
$    s = \sum_{J \subseteq \set{1,\ldots,q}} s_J \theta^J$
where the $s_J$ are smooth functions in $x_1,\ldots,x_p$. Define
\begin{equation} \label{eq:formaltaylor}
    x_A(s)
    = \sum_{\substack{\nu \in \N^p \\ J \subseteq \set{1,\ldots,q}}}
    \frac{1}{\nu!} \frac{\partial^{\abs{\nu}} s_J}{\partial x^{\nu}} \bigg|_{(\red{\X_1},\ldots,\red{\X_p})} \nil{\X}^\nu \T^J \text{.}
\end{equation}
This is the way in which the purely formal expression
\[
    s(x_A) = s(\red{\X_1}+\nil{\X_1},\dots,\red{\X_p}+\nil{\X_p},\T_1,\dots,\T_q)
\]
is usually understood. Eq.\ \eqref{eq:formaltaylor}  has only a
finite number of terms due to the nilpotency of the $\nil{\X_i}$ and
$\T_j$. 
$x_{A}$ is a superalgebra morphism as one can readily check.
\end{proof}

\begin{observation} \label{obs:coordinates}
Let $U$ be a chart in a supermanifold $M$ with local coordinates
$\set{x_i,\theta_j}$.
We have an injective map
\[
        U_A \to A_0^p \times A_1^q,\qquad
        x_A \mapsto (\X_1,\ldots,\X_p,\T_1,\ldots,\T_q) \coloneqq
            \big( x_A(x_1),\ldots, x_A(\theta_q) \big) \text{.}
\]
We can think of it heuristically as the assignment of $A$\nbd valued
coordinates $\set{\X_i,\T_j}$ on $U_A$. As we are going to see in
theorem \ref{theor:azerolinear} the components of the coordinates
$\set{\X_i,\T_j}$, given by $\pair{\smash{a^*_k}}{\X_i}$,
$\pair{\smash{a^*_k}}{\T_j}$ with respect to a basis $\set{a_k}$ of
$A$, are indeed the coordinates of a smooth manifold. The base point
$\red{x_A} \in U$ has coordinates $(\red{\X_1},\ldots,\red{\X_p})$.
In this language, if $\rho \colon A \to B$ is a super Weil algebra
morphism, the corresponding morphism $\funcpt{\rho} \colon M_A \to
M_B$ is ``locally'' given by
$    \rho \times \dots \times \rho \colon A_0^p \times A_1^q \to
B_0^p \times B_1^q$.
This is well defined since $\funcpt{\rho}$ does not change the base
point.

If $M = \R^{p|q}$ we can also consider the slightly different
identification
\[
    \R^{p|q}_A \to (A \otimes \R^{p|q})_0, \qquad
    x_A \mapsto \textstyle\sum_i x_A(e_i^*) \otimes e_i
\]
where $\set{e_1,\ldots,e_{p+q}}$ denotes a homogeneous basis of
$\R^{p|q}$ and $\set{e_1^*,\ldots,e_{p+q}^*}$ its dual basis. Here a
little care is needed. In the literature
the name $\R^{p|q}$ is used for two  different objects:
it may indicate the super vector space $\R^{p|q}=\R^p \oplus \R^q$
or the superdomain $(\R^p,\Cinf_{\R^p} \otimes \extn{q})$. In the
previous equation the first $\R^{p|q}$ is viewed as a superdomain,
while the last as a super vector space. Likewise the $\set{e_i^*}$
are interpreted both as vectors and sections of $\sheaf(\R^{p|q})$.
As we shall see in section
\ref{sec:WB_functor-Sh_emb}, the
functor
$    A \mapsto (A \otimes \R^{p|q})_0$
recaptures all the information about the superdomain $\R^{p|q}$, so
that the two  different ways of looking at $\R^{p|q}$
become identified naturally.
In such identification,
the superdomain morphism $\funcpt{\rho} \colon \R^{p|q}_A \to
\R^{p|q}_B$ corresponds to the super vector space morphism
$    \rho \otimes \id \colon (A \otimes \R^{p|q})_0 \to (B \otimes
\R^{p|q})_0$.
\end{observation}


As we have seen, we can associate to each supermanifold $M$ a
functor
$M_{(\blank)} \colon \SWA \to \Sets$, $A \mapsto M_A$. 
Hence we have a functor:
$    \ber \colon \SMan \to \funct{\SWA}{\Sets}$.
The natural question is whether $\ber$ is a
full and faithful embedding or not.
We are going to show that $\ber$ is not full, in other
words, there are many more natural transformations between
$M_{(\blank)}$ and $N_{(\blank)}$ than those coming from morphisms
from $M$ to $N$.

\medskip

We first want to show that the natural transformations $M_{(\blank)}
\to N_{(\blank)}$ arising from supermanifold morphisms $M\to N$ have
a very peculiar form. Indeed, a morphism $\phi \colon M \to N$ of
supermanifolds induces a natural transformation between the
corresponding functors of $A$\nbd points given by
\[
    \begin{aligned}
        \phi_A \colon M_A &\to N_A, \qquad
        x_A &\mapsto x_A \circ \phi^*
    \end{aligned}
\]
for all super Weil algebras $A$. Let $M = \R^{p|q}$ and $N =
\R^{m|n}$, and denote respectively by $\set{x_i,\theta_j}$ and
$\set{x'_k,\theta'_l}$ two systems of canonical
coordinates over them. With these assumptions, $\phi$ is determined
by the pullbacks of the coordinates of $N$, while the $A$\nbd point
$\phi_A(x_A)$ is determined by
\[
    (\X'_1,\ldots,\X'_{m},\T'_1,\ldots,\T'_{n}) \coloneqq
    \big( x_A \circ \phi^*(x'_1),\ldots,x_A \circ \phi^*(\theta'_{n}) \big) \in A_0^{m} \times A_1^{n} \text{.}
\]
If $(\X_1,\ldots,\X_p,\T_1,\ldots,\T_q)$ denote the images of the
coordinates of $M$ under $x_A$ ($\X_1 = x_A(x_1)$, etc.) and
$\phi^*(x'_k) = \sum_J s_{k,J} \theta^J \in \sheaf(\R^{p|q})_0$,
where the $s_{k,J}$ are functions on $\R^p$, then we have
\begin{equation} \label{eq:nat_tr_from_morph}
    \X'_k = x_A \circ \phi^*(x'_k)
    = \sum_{\substack{\nu \in \N^p \\ J \subseteq \set{1,\ldots,q}}}
    \frac{1}{\nu!} \frac{\partial^{\abs{\nu}} s_{k,J}}{\partial x^{\nu}} \bigg|_{(\red{\X_1},\ldots,\red{\X_p})} \nil{\X}^\nu \T^J
\end{equation}
and similarly for the odd coordinates (see prop.\
\ref{prop:valori_assegnati}). Notice that if we pursue the point of
view of observation \ref{obs:coordinates}, i.~e.\ if we consider
$\set{\X_i,\T_j}$ as $A$\nbd valued coordinates of $\R^{p|q}_A$,
this equation can be read as a coordinate expression for $\phi_A$.

Not all the natural transformations $M_{(\blank)} \to N_{(\blank)}$
arise in this way. This happens also for purely even manifolds, as
we see in the next example.

\begin{example} \label{exampla:counter-example_nat_tr}
Let $M$ and $N$ be two smooth manifolds and let $\phi \colon M \to
N$ be a map (smooth or not). The natural transformation
$\alpha_{(\blank)} \colon M_{(\blank)} \to N_{(\blank)}$,
$\alpha_A(x_A)=\ev_{\phi(\red{x_A})}$,
is not of the form seen above, even if $\phi$ is assumed to be
smooth, while we still have $\phi=\alpha_\R$.
\end{example}

We end this section with a technical result,
essentially due to A. A. Voronov (see \cite{Voronov}),
characterizing all possible natural transformations between the
functors of $A$\nbd points of two superdomains, hence
also those not arising from supermanifold morphisms.

\begin{definition}
Let $U$ be an open subset of $\R^p$. We denote by
$\fseries{p}{q}(U)$ the unital commutative superalgebra of formal
series with $p$ even and $q$ odd generators and coefficients in the
algebra $\functions(U,\R)$ of arbitrary functions on $U$, i.~e.\
$    \fseries{p}{q}(U) \coloneqq$ $\functions(U,\R)[[X_1,\dots,X_p,
\Theta_1,\dots,\Theta_q]]$.
An element $F \in \fseries{p}{q}(U)$ is of the form
 $   F
    = \sum_{\substack{\nu \in {\N}^p \\ J  \subseteq \set{1,\ldots,q}}}
    f_{\nu,J} X^\nu \Theta^J$,
where $f_{\nu,J} \in \functions(U,\R)$ and $\set{X_i}$ and
$\set{\Theta_j}$ are even and odd generators. $\fseries{p}{q}(U)$ is
a graded algebra: $F$ is even (resp.\ odd) if $\abs{J}$ is even
(resp.\ odd) for each term of the sum.
\end{definition}

Let us introduce a partial order between super Weil algebras by
saying that $A' \preceq A$ if and only if $A'$ is a quotient of $A$.

\begin{lemma} \label{lemma:upper_bound}
The family of super Weil algebras is directed, i.~e., if $A_1$ and
$A_2$ are super Weil algebras, then there exists $A$ such that $A_i
\preceq A$.
\end{lemma}

\begin{proof}
In view of lemma \ref{lemma:SWA}, choosing carefully $k, l \in \N$
and $J_1$ and $J_2$ ideals of $\sheaf_{\R^{p|q},0}$, we have $A_i
\isom \sheaf_{\R^{p|q},0} / J_i$.  If $r$ is the maximum between the
heights of $A_1$ and $A_2$, $\maxid_0^{r+1} \subseteq J_1 \cap J_2$.
So $A \isom \sheaf_{\R^{p|q},0} / (J_1 \cap J_2)$ and then it is a
super Weil algebra.
\end{proof}

\begin{proposition} \label{prop:formal_series}
Let $U$ and $V$ be two superdomains in $\R^{p|q}$ and $\R^{m|n}$
respectively. The set of natural transformations in
$\funct{\SWA}{\Sets}$ between $U_{(\blank)}$ and $V_{(\blank)}$ is
in bijective correspondence with the set of elements of the form
\[
    \F = (F_1,\ldots,F_{m+n}) \in \big( \fseries{p}{q}(\topo{U}) \big)_0^m \times \big( \fseries{p}{q}(\topo{U}) \big)_1^n
\]
such that, $F_k = \sum_{\nu,J} f^k_{\nu,J} X^\nu \Theta^J$,
$\big( f^1_{0,\emptyset}(x),\ldots,f^m_{0,\emptyset}(x) \big)
\subseteq \topo{V}$, $\forall x \in \topo{U}$.
\end{proposition}

\begin{proof}
As above, $\R^{p|q}_A$ is identified with $A_0^p \times A_1^q$ and
consequently a map $\R^{p|q}_A \to \R^{m|n}_A$ consists of a list of
$m$ maps $A_0^p \times A_1^q \to A_0$ and $n$ maps $A_0^p \times
A_1^q \to A_1$. In the same way, $U_A$ is identified with $\topo{U}
\times \nil{A_0}^p \times A_1^q$.

Let $\F = (F_1,\ldots,F_{m+n})$ be as in the hypothesis. A formal
series $F_k$ determines a map $\topo{U} \times \nil{A_0}^p \times
A_1^q \subseteq A_0^p \times A_1^q \to A$ in a natural way, defining
\[
    F_k(\X_1,\ldots,\X_p,\T_1,\ldots,\T_q)
    \coloneqq \sum_{\substack{\nu \in \N^p \\ J \subseteq \set{1,\ldots,q}}}
    f^k_{\nu,J}(\red{\X_1},\ldots,\red{\X_p}) \nil{\X}^\nu \T^J \text{.}
\]
The parity of its image is the same as that of $F_k$. Then, in view
of the restrictions imposed on the first $m$, $F_k$ given by the
equation above, $\F$ determines a map $U_A \to V_A$ and, varying $A
\in \SWA$, a natural transformation $U_{(\blank)} \to V_{(\blank)}$,
as it is easily checked.

Let us now suppose now that $\alpha_{(\blank)} \colon U_{(\blank)}
\to V_{(\blank)}$ is a natural transformation. We will see that it
is determined by an unique $\F$ in the way just explained.

Let $A$ be a super Weil algebra of height $r$ and
\[
    x_A = (\red{\X_1} + \nil{\X_1},\ldots,\red{\X_p} + \nil{\X_p},\T_1,\ldots,\T_q)
    \in A_0^p \times A_1^q \isom \R^{p|q}_A
\]
with $\red{x_A} \in \topo{U}$. Let us consider the super Weil
algebra
\begin{equation} \label{eq:hatA}
    \hat{A} \coloneqq \big( \R[z_1,\ldots,z_p] \otimes \ext{\zeta_1,\ldots,\zeta_q} \big) / \maxid^{s}
\end{equation}
with $s > r$ ($\maxid$ is as usual the maximal ideal of polynomials
without constant term) and the $\hat{A}$\nbd point
$    y_{\red{x_A}} \coloneqq (\red{\X_1} + z_1,\ldots,\red{\X_1} +
z_p,\zeta_1,\ldots,\zeta_q)
    \in \hat{A}_0^p \times \hat{A}_1^q \isom \R^{p|q}_{\hat{A}}$.

A homomorphism between two super Weil algebras is clearly fixed by
the images of a set of generators, but this assignment must be
compatible with the relations between the generators. The following
assignment is possible due to the definition of $\hat{A}$. If
$\rho_{x_A} \colon \hat{A} \to A$ denotes the map
$\rho_{x_A}(z_i)=\nil{\X_i}$, $\rho_{x_A}(\zeta_j)= \T_j$,
then clearly $\funcpt{\rho_{x_A}}(y_{\red{x_A}}) = x_A$.

Let ${(\alpha_{\hat{A}})}_k$ with  $1 \leq k \leq m+n$ be a
component of $\alpha_{\hat{A}}$, and let
$    {(\alpha_{\hat{A}})}_k(y_{\red{x_A}}) = \sum_{\nu,J}
a^k_{\nu,J}(\red{x_A}) z^\nu \zeta^J$
with $a^k_{\nu,J}(\red{x_A}) \in\R$ and $\big(
a^1_{0,\emptyset}(\red{x_A}),\ldots,a^m_{0,\emptyset}(\red{x_A})
\big) \in \topo{V}$; the sum is on $\abs{J}$ even (resp.\ odd), if
$k \leq m$ (resp.\ $k
> m$). Due to the functoriality of $\alpha_{(\blank)}$
\[
    {(\alpha_A)}_k (x_A)
    = {(\alpha_{A})}_k \circ \funcpt{\rho_{x_A}} (y_{\red{x_A}})
    = \rho_{x_A} \circ {(\alpha_{\hat{A}})}_k (y_{\red{x_A}})
    = \sum_{\nu,J} a^k_{\nu,J}(\red{x_A}) \nil{\X}^\nu \T^J \text{,}
\]
so that there exists a non unique $\F$ such that $\F(x_A) =
\alpha_A(x_A)$. Moreover $\F(x_{A'}) = \alpha_{A'}(x_{A'})$ for each
$A' \preceq A$ and $x_{A'} \in U_{A'}$ (it is sufficient to use the
projection $A \to A'$). If $\F'$ is another list of formal series
with this property, there exists a super Weil algebra $A''$ such
that $\F(x_{A''}) \neq \F'(x_{A''})$ for some $x_{A''} \in U_{A''}$.
Indeed if a component $F_k$ differs in $f^k_{\nu,J}$, it is
sufficient to consider $A'' \coloneqq \rspoly{p}{q} / \maxid^{s}$
with $s > \max \big( \abs{\nu},q \big)$.
\end{proof}

\section{The Weil--Berezin functor and the Schwarz embedding}
\label{sec:WB_functor-Sh_emb}

In the previous section we saw that the functor
    $\ber \colon \SMan \to \funct{\SWA}{\Sets}$
does not define a full and faithful embedding of $\SMan$ in
$\funct{\SWA}{\Sets}$. Roughly speaking, the root of such a
difficulty can be traced to the fact that the functor $\ber(M)
\colon \SWA \to \Sets$ looks only to the local structure of the
supermanifold $M$, hence it loses all the global information.
The following heuristic argument gives a hint on how we can overcome
such problem.

It is well known (see, for example, \cite[\S~1.7]{DM}) that if
$V=V_0\oplus V_1$ and $W=W_0\oplus W_1$ are super vector spaces, there
is a bijective correspondence between linear maps $V\to W$
and functorial families of $\Lambda_0$\nbd linear maps between
$(\Lambda \otimes V)_0$ and $(\Lambda \otimes W)_0$, for each
Grassmann algebra $\Lambda$. This result goes under the name of
\emph{even rule principle}. Since vector spaces are local models for
manifolds, the even rule principle seems to suggest that each $M_A$
should be endowed with a local structure of $A_0$\nbd module. This
vague idea is made precise with the introduction of the category
$\cAoMan$ of $A_0$\nbd smooth manifolds.

\begin{definition} \label{def:Az_man}
Fix an even commutative finite dimensional algebra $A_0$ and let $L$
be an $A_0$\nbd-module, finite dimensional
as a real vector space. Let $M$ be a manifold. An
\emph{$L$\nbd chart} on $M$ is a pair $(U,h)$ where $U$ is open in
$M$ and $h \colon U \to L$ is a diffeomorphism onto its image. $M$
is an \emph{$A_0$\nbd manifold} if it admits an $L$\nbd atlas. By
this we mean a family $\set{(U_i,h_i)}_{i \in \cA}$ where
$\set{U_i}$ is an open covering of $M$ and each $(U_i,h_i)$ is an
$L$\nbd chart, such that the differentials
\[
    d(h_i \circ h_j^{-1})_{h_j(x)} \colon T_{h_j(x)} (L) \isom L \to L \isom T_{h_i(x)} (L)
\]
are isomorphisms of $A_0$\nbd modules for all $i$, $j$ and $x \in
U_i \cap U_j$.

If $M$ and $N$ are $A_0$\nbd manifolds, a \emph{morphism} $\phi
\colon M \to N$ is a smooth map whose differential is $A_0$\nbd
linear at each point. We also say that such morphism is
\emph{$A_0$\nbd smooth}. We denote by $\AoMan$ the category of
$A_0$\nbd manifolds.

We define also the category  $\cAoMan$ in the following way. The
objects of  $\cAoMan$ are manifolds over generic finite dimensional
commutative algebras. The morphisms in the category are defined as
follows. Denote by $A_0$ and $B_0$  two commutative finite
dimensional algebras, and let $\rho \colon A_0\to B_0$ be  an
algebra morphism. Suppose  $M$ and $N$ are $A_0$ and $B_0$ manifolds
respectively, we say that a morphism $\phi \colon M \to N$ is
\emph{$\rho$\nbd smooth} if $\phi$ is smooth and
$    (d\phi)_x(a v) = \rho(a)(d\phi)_x(v)$
for each $x \in M$, $v \in T_x(M)$, and $a \in A_0$ (see \cite{Shurygin}
for more details).
\end{definition}


The above definition is motivated by the following theorems. In
order to ease the exposition we first give the statements of the
results postponing their proofs to later.

\begin{theorem} \label{theor:azerolinear}
Let $M$ be a smooth  supermanifold, and let $A \in \SWA$.\\
1. $M_A$ can be endowed with a unique $A_0$\nbd manifold structure
such
    that, for each open subsupermanifold $U$ of $M$ and $s \in \sheaf_M(U)$
    the map defined by $\hat{s} \colon U_A \to A$, $ x_A \mapsto x_A(s)$,
    is $A_0$\nbd smooth. \\
2. If $\phi \colon M\to N$ is a supermanifold
    morphism, then $\phi_A \colon M_A \to N_A$,
$x_A \mapsto x_A \circ \phi^*$
    is an $A_0$\nbd smooth morphism. \\
3. If $B$ is another super Weil algebra and $\rho \colon A\to B$ is
an
    algebra morphism, then $\funcpt{\rho} \colon M_A \to M_B$,
 $x_A \mapsto \rho \circ x_A$
    is a $\restr{\rho}{A_0}$\nbd smooth map.
\end{theorem}

The above theorem says that supermanifolds morphisms give rise to
morphisms in the $\AoMan$ category.  From this point of view the
next definition is quite natural.

\begin{definition} \label{def:Az_nat_transf}
We call $\dfunct{\SWA}{\cAoMan}$ the subcategory of
$\funct{\SWA}{\cAoMan}$ whose objects are the same and whose
morphisms $\alpha_{(\blank)}$ are the natural transformations $\cF
\to \cG$, with $\cF, \cG \colon \SWA \to \cAoMan$, such that
$    \alpha_A \colon \cF(A) \to \cG(A)$
is $A_0$\nbd smooth for each $A\in \SWA$.
\end{definition}

Theorem \ref{theor:azerolinear} allows us to give more structure to
the image category of the functor of $A$\nbd points. More
precisely we have  the following definition, which is the central
definition in our treatment of the local functor of points.

\begin{definition}
Let $M$ be a supermanifold. We define the \emph{Weil--Berezin
functor} of $M$ as
\begin{equation} \label{eq:WB_functor}
\begin{aligned}
        M_{(\blank)} \colon \SWA \to \cAoMan, & \quad
        A \mapsto M_A \\
\end{aligned}
\end{equation}
and the \emph{Schwarz embedding} as
\begin{equation} \label{eq:WB_functor2}
\begin{aligned}
    \sch \colon \SMan \to \dfunct{\SWA}{\cAoMan}, & \quad
    M \mapsto M_{(\blank)} \text{.}
\end{aligned}
\end{equation}
\end{definition}

We can now state one of the main results in this paper.

\begin{theorem} \label{theor:full_and_faithful}
$\sch$ is a full and faithful embedding, i.~e.\ if $M$ and $N$ are
two supermanifolds, and $M_{(\blank)}$ and $N_{(\blank)}$ their
Weil--Berezin functors, then
\[
    \Hom_\SMan(M,N) \isom \Hom_{ \dfunct{\SWA}{\cAoMan}}
    (M_{(\blank)},N_{(\blank)}) \text{.}
\]
\end{theorem}

\begin{observation}
If we considered the bigger category $\funct{\SWA}{\cAoMan}$ instead
of $\dfunct{\SWA}{\cAoMan}$, the above theorem is no longer true. In
example \ref{exampla:counter-example_nat_tr} we examined a natural
transformation between functors from $\SWA$ to $\Sets$, which does
not come from a supermanifold morphism. If, in the same example,
$\phi$ is chosen to be smooth, we obtain a morphism in
$\funct{\SWA}{\cAoMan}$ that is not in $\dfunct{\SWA}{\cAoMan}$.
Indeed, it is not difficult to check that if $\pi_A \colon A \to A$
is given by $a \mapsto \red{a}$, then $\alpha_A$ (in the example) is
$\pi_{A_0}$\nbd linear.
\end{observation}

We now examine the proofs of theorems \ref{theor:azerolinear} and
\ref{theor:full_and_faithful}. First we need to prove theorem
\ref{theor:full_and_faithful} in the case of two superdomains $U$
and $V$ in $\R^{p|q}$ and $\R^{m|n}$ respectively (lemma
\ref{lemma:A0_smooth_nat_tr}). As usual, if $A$ is a super Weil
algebra, $U_A$ and $V_A$ are identified with $\topo{U} \times
\nil{A_0}^p \times A_1^q$ and $\topo{V} \times \nil{A_0}^{m} \times
A_1^{n}$ (see observation \ref{obs:coordinates}). Then they have a
natural structure of open subsets of $A_0$\nbd modules. Next
lemma is due to A. A. Voronov in \cite{Voronov} and it is the local
version of theorem \ref{theor:full_and_faithful}.

\begin{lemma} \label{lemma:A0_smooth_nat_tr}
A natural transformation $\alpha_{(\blank)} \colon U_{(\blank)} \to
V_{(\blank)}$ comes from a supermanifold morphism $U \to V$ if and
only if $\alpha_A \colon U_A \to V_A$ is $A_0$\nbd smooth for each
$A$.
\end{lemma}

\begin{proof}
Due to prop.\ \ref{prop:formal_series} we know that
$\alpha_{(\blank)}$ is determined by $m$ even and $n$ odd formal
series of the form $F_k = \sum_{\nu,J} f^k_{\nu,J} X^\nu \Theta^J$
with $f^k_{\nu,J}$ arbitrary functions in $p$ variables satisfying
suitable conditions. Moreover as we have
seen in the discussion before example
\ref{exampla:counter-example_nat_tr} a supermanifold morphism $\phi
\colon U \to V$ gives rise to a natural transformation $\phi_A
\colon U_A \to V_A$ whose components are of the form of eq.\
\eqref{eq:nat_tr_from_morph}. Let us suppose that $\alpha_A$ is
$A_0$\nbd smooth. This clearly happens if and only if all its
components are $A_0$\nbd smooth and the smoothness request for all
$A$ forces all coefficients $f^k_{\nu,J}$ to be smooth. Let
${(\alpha_A)}_k$ be the $k$\nbd th component of $\alpha_A$ and let
$i \in \set{1,\ldots,p}$. We want to study
$        \omega \colon A_0 \to A_j$, $
\omega(\X_i)\coloneqq{(\alpha_A)}_k(\X_1,\ldots,\X_i,\ldots,
\X_p,\T_1,\ldots,\T_q)$,
supposing the other coordinates fixed ($j = 0$ if $1 \leq k \leq p$
or $j = 1$ if $p < k \leq p+q$). Since $\nil{\X_i} \in A_0$ commutes
with all elements of $A$,
\begin{equation} \label{eq:omega}
    \omega(\X_i) = \sum_{t \geq 0} a_t(\red{\X_i}) \nil{\X_i}^t, \quad
    a_t(\red{\X_i}) \coloneqq \sum_{\substack{\nu,J \\ \nu_i=t}} f^k_{\nu,J}(\red{\X_1},\ldots,\red{\X_i},\ldots,\red{\X_p}) \nil{\X}^{(\nu-t\delta_i)} \T^J
\end{equation}
($t\delta_i$ is the element of $\N^p$ with $t$ at the $i$\nbd th
component and $0$ elsewhere). If $\Y = \red{\Y} + \nil{\Y} \in A_0$
and $\omega$ is $A_0$\nbd smooth
\begin{equation} \label{eq:Alin_diff_1}
    \omega(\X_i + \Y) - \omega(\X_i) = d\omega_{\X_i}(\Y) + o(\Y) = (\red{\Y} + \nil{\Y}) d\omega_{\X_i}(1_A) + o(\Y)
\end{equation}
(where $1_A$ is the unit of $A$). On the other hand, from eq.\
\eqref{eq:omega} and defining
\begin{equation} \label{eq:a'_t}
    a_t'(\red{\X_i}) \coloneqq \sum_{\substack{\nu,J \\ \nu_i=t}}
    \partial_i f^k_{\nu,J}(\red{\X_1},\ldots,\red{\X_i},\ldots,\red{\X_p}) \nil{\X}^{(\nu-t\delta_i)} \T^J
\end{equation}
($\partial_i$ denotes the partial derivative respect to the $i$\nbd
th variable), we have
\begin{equation} \label{eq:Alin_diff_2}
    \begin{aligned}
        \omega(\X_i + \Y) - \omega(\X_i)
        &= \sum_{t \geq 0} a_t (\red{\X_i} + \red{\Y}) (\nil{\X_i} + \nil{\Y})^t - \sum_{t \geq 0} a_t(\red{\X_i}) \nil{\X_i}^t \\
        &= \sum_{t \geq 0} \left( a_t'(\red{\X_i}) \red{\Y} \nil{\X_i}^t +
        a_t(\red{\X_i}) t \nil{\X_i}^{t-1} \nil{\Y} + o(\Y) \right) \\
        &= \red{\Y} \sum_{t \geq 0} a_t'(\red{\X_i}) \nil{\X_i}^t + \nil{\Y} \sum_{t \geq 0} (t+1) a_{t+1}(\red{\X_i}) \nil{\X_i}^t + o(\Y) \text{.}
    \end{aligned}
\end{equation}
Thus, comparing eq.\ \eqref{eq:Alin_diff_1} and
\eqref{eq:Alin_diff_2}, we get that the identity
\[
    (\red{\Y} + \nil{\Y}) d\omega_{\X_i}(1_A)
    = \red{\Y} \sum_{t \geq 0} a_t'(\red{\X_i}) \nil{\X_i}^t + \nil{\Y} \sum_{t \geq 0} (t+1) a_{t+1}(\red{\X_i}) \nil{\X_i}^t
\]
must hold and, consequently, also the following relations must be
satisfied:
\begin{align*}
    \sum_{t \geq 0} a_t'(\red{\X_i}) \nil{\X_i}^t
    &= \sum_{t \geq 0} (t+1) a_{t+1}(\red{\X_i}) \nil{\X_i}^t \\
\intertext{and then, from eq.\ \eqref{eq:omega} and \eqref{eq:a'_t},}
    \sum_{\nu,J} \partial_i f^k_{\nu,J}(\red{\X_1},\ldots,\red{\X_p}) \nil{\X}^\nu \T^J
    &= \sum_{\nu,J} (\nu_i + 1) f^k_{\nu + \delta_i,J}(\red{\X_1},\ldots,\red{\X_p}) \nil{\X}^\nu \T^J
    \text{.}
\end{align*}

Let us fix $\nu \in \N^p$ and $J \subseteq \set{1,\ldots,q}$. If $A
= \rspoly{p}{q}/\maxid^s$ with $s > \max(\abs{\nu}+1,q)$ ($\maxid$
is as usual the maximal ideal of polynomials without constant term),
we note that necessarily, due to the arbitrariness of
$(\X_1,\ldots,\T_q)$,
\[
    \partial_i f^k_{\nu,J} = (\nu_i + 1) f^k_{\nu + \delta_i,J}
\]
and, by recursion, ${(\alpha_A)}_k$ is of the form of
\eqref{eq:nat_tr_from_morph} with $s_{k,J} = f^k_{0,J}$.

Conversely, let ${(\alpha_A)}_k$ be of the form of eq.\
\eqref{eq:nat_tr_from_morph}. By linearity, it is $A_0$\nbd linear
if and only if it is $A_0$\nbd linear in each variable. It is
$A_0$\nbd linear in the even variables for what has been said above
and in the odd variables since it is polynomial in them.
\end{proof}

In particular the above discussion shows also that any
superdiffeomorphism $U \to U$ gives rise, for each $A$, to an
$A_0$\nbd smooth diffeomorphism $U_A \to U_A$ and then each $U_A$
admits a canonical structure of $A_0$\nbd manifold.

We now use the results obtained for superdomains in order to prove
theorems \ref{theor:azerolinear} and \ref{theor:full_and_faithful}
in the general supermanifold case.




\begin{proof}[Proof of theorem \ref{theor:azerolinear}]
Let $\set{(U_i, h_i)}$ be an atlas over $M$ and $p|q$ the dimension
of $M$. Each chart $(U_i,h_i)$ of such an atlas induces a chart
$\big((U_i)_A, (h_i)_A\big)$,
over $M_A$ given by
$        (h_i)_A \colon (U_i)_A \to \R^{p|q}_A$, $        x_A
\mapsto x_A \circ h_i^*$.
The coordinate changes are easily checked to be given, with some
abuse of notation, by $(h_{i} \circ h_{j}^{-1})_A$, which are
$A_0$\nbd smooth due to lemma \ref{lemma:A0_smooth_nat_tr}. The
uniqueness of the $A_0$\nbd manifold structure is clear. This proves
the first point. The other two points concern only the local
behavior of the considered maps and are clear in view of lemma
\ref{lemma:A0_smooth_nat_tr} and obs. \ref{obs:coordinates}.
\end{proof}

\begin{proof}[Proof of theorem \ref{theor:full_and_faithful}]
Lemma \ref{lemma:A0_smooth_nat_tr} accounts for the case in which
$M$ and $N$ are superdomains. For the general case, let us suppose
we have
$    \alpha \in \Hom_{\dfunct{\SWA}{\cAoMan}} \big( M_{(\blank)} ,
N_{(\blank)} \big)$.
Fixing a suitable atlas of both supermanifolds, we obtain, in view
of lemma \ref{lemma:A0_smooth_nat_tr},  a family of local morphisms.
Such a family will give a morphism $M \to N$ if and only if they do
not depend on the choice of the coordinates. Let us suppose that $U$
and $V$ are open subsupermanifolds of $M$ and $N$ respectively, $U
\isom \R^{p|q}$, $V \isom \R^{m|n}$, such that $\alpha_\R(\topo{U})
\subseteq \topo{V}$, and
$    h_i \colon U \to \R^{p|q}$, $    k_i \colon V \to \R^{m|n}$, $
i = 1,2$
are two different choices of coordinates on $U$ and $V$
respectively. The natural transformations
\[
    (\hat{\phi}_i)_{(\blank)} \coloneqq \left( k_i \right)_{(\blank)}
        \circ \restr{\left( \alpha_{(\blank)} \right)}{U_{(\blank)}}
        \circ \left( h_i^{-1} \right)_{(\blank)}
        \colon \R^{p|q}_{(\blank)} \to \R^{m|n}_{(\blank)}
\]
give rise to two morphisms $\hat{\phi}_i \colon \R^{p|q} \to
\R^{m|n}$. If
 $  \phi_i \coloneqq k_i^{-1} \circ \hat{\phi}_i \circ h_i \colon U \to V$,
we have $\phi_1 = \phi_2$ since $(\phi_i)_{(\blank)} = \restr{\left(
\alpha_{(\blank)} \right)}{U_{(\blank)}}$ and two morphisms that
give rise to the same natural transformation on a superdomain are
clearly equal.
\end{proof}

Next proposition states that the Schwarz embedding
 that the Schwarz
embedding preserves products and, in consequence, group objects.

\begin{proposition} \label{prop:s-emb}
For all supermanifolds $M$ and $N$,

\medskip

\centerline{$    \sch(M \times N) \isom \sch(M) \times \sch(N) \text{.}$}

\medskip

Moreover $\sch(\R^{0|0})$ is a terminal object in the category
$\dfunct{\SWA}{\cAoMan}$.
\end{proposition}

\begin{proof}
The fact that $(M \times N)_A \isom M_A \times N_A$ for all $A$ can
be checked easily. Indeed, let $z_A\in (M\times N)_A$ with
$\red{z_A}=(x,y)$, we have that $\sheaf(M)$ and
$\sheaf(N)$ naturally inject in $\sheaf(M\times N)$.
 Hence $z_A$ defines,
by restriction, two $A_0$\nbd points $x_A\in M_A$ and $y_A\in N_A$.
Using prop.\ \ref{prop:valori_assegnati} and rectangular coordinates
over $M\times N$ it is easy to check that such a correspondence is
injective, and is also a natural transformation. Conversely, if
$x_A\in M_{A}$ is near $x$ and $y_A\in N_A$
is near $y$ (see obs. \ref{obs:stalk}), they define a map $z_A \colon
\sheaf(M\times N) \to A$ through $z_A(s_1 \otimes
s_2)=x_A(s_1)\cdot y_A(s_2)$. Using again prop.\
\ref{prop:valori_assegnati}, it is not difficult to check that this
requirement uniquely determines a superalgebra morphism $\sheaf{(M
\times N)} \to A$ and
that this correspondence defines an inverse for the morphism
$(M\times N)_{(\blank)}\to M_{(\blank)}\times N_{(\blank)}$ defined
above.
Along the same lines we see that a  similar condition for
the morphisms holds. Finally $\sch(\R^{0|0})$ is a terminal object,
since $\R^{0|0}_A = \R^0$ for all $A$.
\end{proof}





\begin{corollary}
The Weil--Berezin functor of a super Lie group (i.~e.\ a
group object in the category of supermanifolds) takes values in
the category of $A_0$\nbd smooth Lie groups.
\end{corollary}





We now turn to representability questions.

\begin{definition}
We say that a functor
 $   \cF \colon \SWA \to \cAoMan$
is \emph{representable} if there exists a supermanifold $M_\cF$ such
that $\cF \isom (M_\cF)_{(\blank)}$ in $\dfunct{\SWA}{\cAoMan}$.
\end{definition}

Notice that we are abusing the category terminology, that considers
a functor $\cF$ to be representable if and only if $\cF$ is
isomorphic to the $\Hom$ functor.

\medskip

Due to theorem \ref{theor:full_and_faithful}, if a functor $\cF$ is
representable, then the supermanifold $M_\cF$ is unique up to
isomorphism.




\medskip

Since $\cF(\R)$ is a manifold, we can consider an open set $U
\subseteq \cF(\R)$. If $A$ is a super Weil algebra and
$\funcpt{\pr_A} \coloneqq \cF(\pr_A)$, where $\pr_A$ is the
projection $A \to \R$, $\funcpt{\pr_A}^{-1}(U)$ is an open $A_0$\nbd
submanifold of $\cF(A)$. Moreover, if $\rho \colon A \to B$ is a
superalgebra map, since $\pr_B \circ \rho = \pr_A$, $\funcpt{\rho}
\coloneqq \cF(\rho)$ can be restricted to
$    \subfunc{\funcpt{\rho}}{\funcpt{\pr_A}^{-1}(U)} \colon
\funcpt{\pr_A}^{-1}(U) \to \funcpt{\pr_B}^{-1}(U)$.
We can hence define the functor
$        \subfunc{\cF}{U} \colon \SWA \to \cAoMan$, $        A
\mapsto \funcpt{\pr_A}^{-1}(U)$, $ \rho
\mapsto \subfunc{\funcpt{\rho}}{\funcpt{\pr_A}^{-1}(U)}$.

\begin{proposition}[Representability]
A functor
$    \cF \colon \SWA \to \cAoMan$
is representable if and only if there exists an open cover
$\set{U_i}$ of $\cF(\R)$ such that $\subfunc{\cF}{U_i} \isom
(\bV_i)_{(\blank)}$ with $\bV_i$ superdomains in a fixed $\R^{p|q}$.
\end{proposition}

\begin{proof}
The necessity is clear due to the very definition of supermanifold.
Let us prove sufficiency. We have to build a supermanifold structure
on the topological space $\topo{\cF(\R)}$. Let us denote by
$(h_i)_{(\blank)} \colon \cF_{U_i} \to (\bV_i)_{(\blank)}$ the
natural isomorphisms in the hypothesis. On each $U_i$, we can put a
supermanifold structure $\bU_i$, defining the sheaf $\sheaf_{\bU_i}
\coloneqq [(h_i^{-1})_\R]_* \sheaf_{\bV_i}$. Let $k_i$ be the
isomorphism $\bU_i \to \bV_i$ and $(k_i)_{(\blank)}$ the
corresponding natural transformation. If $U_{i,j} \coloneqq U_i \cap
U_j$, consider the natural transformation $(h_{i,j})_{(\blank)}$
defined by the composition
\[
    (k_i^{-1})_{(\blank)} \circ (h_i)_{(\blank)} \circ (h_j^{-1})_{(\blank)} \circ (k_j)_{(\blank)}
    \colon (U_{i,j},\restr{\sheaf_{\smash{\bU_j}}}{U_{i,j}})_{(\blank)}
    \to (U_{i,j},\restr{\sheaf_{\smash{\bU_i}}}{U_{i,j}})_{(\blank)}
\]
where in order to avoid heavy notations we didn't explicitly
indicate the appropriate restrictions. Each $(h_{i,j})_{(\blank)}$
is a natural isomorphism in $\dfunct{\SWA}{\cAoMan}$ and, due to
lemma \ref{lemma:A0_smooth_nat_tr}, it gives rise to a supermanifold
isomorphism
$    h_{i,j} \colon$
$(U_{i,j},\restr{\sheaf_{\smash{\bU_j}}}{U_{i,j}}) \to
(U_{i,j},\restr{\sheaf_{\smash{\bU_i}}}{U_{i,j}})$.
The $h_{i,j}$ satisfy the cocycle conditions $h_{i,i} = \id$ and
$h_{i,j} \circ h_{j,k} = h_{i,k}$ (restricted to $U_i \cap U_j \cap
U_k$). This follows from the analogous conditions satisfied by
$(h_{i,j})_A$ for each $A\in\SWA$. The supermanifolds $\bU_i$ can
hence be glued (for more information about the construction of a
supermanifold by gluing see for example \cite[ch.~2]{DM} or
\cite[\S~4.2]{Varadarajan}). Denote by $M_\cF$ the manifold thus
obtained.
 Moreover it is clear that $\cF$
is represented by the supermanifold $M_\cF$. Indeed, one can check
that the various $(h_i)_{(\blank)}$ glue together and give a natural
isomorphism $h_{(\blank)} \colon \cF \to (M_\cF)_{(\blank)}$.
\end{proof}

\begin{remark}
The supermanifold $M_\cF$ admits a more synthetic characterization.
In fact it is easily seen that $\topo{M_\cF} \coloneqq
\topo{\cF(\R)}$ and
$    \sheaf_{M_\cF}(U) \coloneqq$ $ \Hom_{\dfunct{\SWA}{\cAoMan}}
\big( \subfunc{\cF}{U}, \R^{1|1}_{(\blank)} \big)$.
\end{remark}


We end this section giving a
brief exposition of
the original approach of A. S. Schwarz and A. A. Voronov (see
\cite{Schwarz,Voronov}). In their work they considered only
Grassmann algebras instead of all super Weil algebras. There are
some advantages in doing so: Grassmann algebras are many fewer,
moreover, as we noticed in remark \ref{remark:localalg}, they are
the sheaf of the super domains $\R^{0|q}$ and so the restriction to
Grassmann algebras of the local functors of points can be considered
as a true restriction of the functor of points. Finally the use of
Grassmann algebras is also used by A. S. Schwarz to formalize the
language commonly used in physics.

On the other hand the use of super Weil algebras has the advantage
that we can perform differential calculus on the Weil--Berezin
functor as we shall see in section \ref{sec:diff_calc}. Indeed
prop.\ \ref{prop:distributions} is valid only for the Weil--Berezin
functor approach, since not every point supported distribution can
be obtained using only Grassmann algebras. Also theorem
\ref{theor:transitivity} and its consequences are valid only in this
approach, since purely even Weil algebras are considered.

\medskip

If $M$ is a supermanifold and $\Gras$ denotes the category of finite
dimensional Grassmann algebras, we can consider the two
functors
\begin{align*}
    &\Gras \to \Sets, \; \Lambda \mapsto M_\Lambda &
    &\text{and} &
    &\Gras \to \cAoMan, \; \Lambda \mapsto M_\Lambda
\end{align*}
in place of those already introduced in the context of $A$-points.
As in the case of $A$\nbd points, with a slight abuse of notation we
denote by $M_\Lambda$ the $\Lambda$\nbd points for each of the two
different functors. What we have seen previously still remains valid in this setting, provided we substitute
systematically $\SWA$ with $\Gras$; in particular theorems
\ref{theor:azerolinear} and \ref{theor:full_and_faithful} still hold
true. They are based on prop.\ \ref{prop:formal_series} and lemma
\ref{lemma:A0_smooth_nat_tr} that we state here in their original
formulation as it is contained in \cite{Voronov}.

\begin{proposition}
The set of natural transformations between $\Lambda \mapsto
\R^{p|q}_{\Lambda}$ and $\Lambda \mapsto \R^{m|n}_{\Lambda}$ is in
bijective correspondence with
$    \big( \fseries{p}{q}(\R^p) \big)_0^m \times \big(
\fseries{p}{q}(\R^p) \big)_1^n$.
A natural transformation comes from a supermanifold morphism
$\R^{p|q} \to \R^{m|n}$ if and only if it is $\Lambda_0$\nbd smooth
for each Grassmann algebra $\Lambda$.
\end{proposition}

\begin{proof}
See proofs of prop.\ \ref{prop:formal_series} and lemma
\ref{lemma:A0_smooth_nat_tr}. The only difference is in the first
proof. Indeed the algebra \eqref{eq:hatA} is not a Grassmann
algebra. So, if $A = \extn{n} = \ext{\epsilon_1,\ldots,\epsilon_n}$,
we have to consider
$    \hat{A} \coloneqq \extn{2p(n-1)+q} =
\ext{\eta_{i,a},\xi_{i,a},\zeta_j}$
($1 \leq i \leq p$, $1 \leq j \leq q$, $1 \leq a \leq n-1$). A
$\extn{n}$\nbd point can be written as
\[
    x_{\extn{n}} = \left( u_1 + \sum_{a<b} \epsilon_a \epsilon_b k_{1,a,b},\ldots,u_p + \sum_{a<b} \epsilon_a \epsilon_b k_{p,a,b},
    \kappa_1,\ldots,\kappa_q \right)
\]
with $u_i \in \R$, $k_{i,a,b} \in (\extn{n})_0$ and $\kappa_j \in
(\extn{n})_1$. Its image under a natural transformation can be
obtained taking the image of the $\extn{2p(n-1)+q}$\nbd point
\[
    y_{\red{x_{\extn{n}}}} \coloneqq \left( u_1 + \sum_{a=1}^{n-1} \eta_{1,a}\xi_{1,a},\ldots,u_p + \sum_{a=1}^{n-1} \eta_{p,a}\xi_{p,a},\zeta_1,\ldots,\zeta_q \right)
\]
and applying the map $\extn{2p(n-1)+q} \to \extn{n}$, $\eta_{i,a}
\mapsto \epsilon_a$, $\xi_{i,a} \mapsto \textstyle\sum_{b>a}
\epsilon_b k_{i,a,b}$, $\zeta_j \mapsto
\kappa_j$
to each component. The nilpotent part of each even component of
$y_{\red{x_{\extn{n}}}}$ can be viewed as a formal scalar product
$    (\eta_{i,1},\ldots,\eta_{i,n-1}) \cdot
(\xi_{i,1},\ldots,\xi_{i,n-1})$ $    = \sum_{a=1}^{n-1}
\eta_{i,a}\xi_{i,a}$.
%
This is stable under formal rotations and the same must be for its
image. So $\eta_{i,a}$ and $\xi_{i,a}$ can occur in the image only
as a polynomial in $\sum_a \eta_{i,a} \xi_{i,a}$. In other words the
image of $y_{\red{x_{\extn{n}}}}$ (and then of $x_{\extn{n}}$) is
polynomial in the nilpotent part of the coordinates.
\end{proof}

\section{Applications to differential calculus} \label{sec:diff_calc}

In this section we discuss some aspects of super differential
calculus on supermanifolds using the language of the Weil--Berezin
functor. In particular we establish a relation between the
$A$\nbd points of a supermanifold $M$ and the finite support
distributions over it, which
play a crucial role in Kostant's
seminal approach to supergeometry.
We also prove the super version of the Weil transitivity
theorem, which is a key tool for the study of the infinitesimal
aspects of supermanifolds.



\medskip

Let $(\topo{M},\sheaf_M)$ be a supermanifold of dimension
$p|q$ and $x \in \topo{M}$. As in \cite[\S~2.11]{Kostant}, let us
consider the distributions with support at $x$.
In what follows we make a full use of obs. \ref{obs:stalk}
which allows us to view any $x_A \in M_A$ as a map
$x_A: \sheaf_{M, \red{x_A}} \longrightarrow A$.

\begin{definition}
Let $\sheaf(M)'$ be the algebraic dual of the superalgebra of global
sections of $M$. The
\emph{distributions with finite support } over $M$ are defined as:
\[
    \sheaf{(M)}^\ast \coloneqq \Set{v \in \sheaf({M})' |
v \big(J\big)=0, \, \mbox{ with } J \mbox{ ideal of finite codimension}  }
\]
We define \emph{the distribution of order $k$, with support at $x\in \red{M}$}
and the \emph{distributions with support at $x$} as follows:

\medskip

\centerline{$\distr[k]{M}{x}\coloneqq \Set{v\in \sheaf({M})' |
v \big( \maxid_{M,x}^k\big)=0}, \qquad
\distr{M}{x} \coloneqq \bigcup_{k = 0}^\infty \distr[k]{M}{x} \text{,}$
}

\medskip

\noindent where $\maxid_{M,x}$ denotes the maximal ideal of sections whose
evaluation at $x$ is zero.
Clearly $\distr[k]{M}{x}
\subseteq \distr[k+1]{M}{x}$.
\end{definition}

\begin{observation}
 If
$x_1,\ldots,x_p,\theta_1,\ldots,\theta_q$ are coordinates in a
neighbourhood of $x$, a distribution of order $k$ is of the form
\[
    v = \sum_{\substack{\nu \in \N^p \\
    J \subseteq \set{1,\ldots,q} \\
    \abs{\nu} + \abs{J} \leq k}} a_{\nu,J} \;
    \ev_x \frac{\partial^{\abs{\nu}}}{\partial x^\nu}
    \frac{\partial^{\abs{J}}}{\partial \theta^J}
\]
with $a_{\nu,J} \in {\R}$. This is immediate since 
$    \distr[k]{M}{x}
       \isom \sheaf[C]^{\infty,*}_{M,x} \otimes {\Lambda(\theta_1,\dots,\theta_q)}^*$
and $\sheaf[C]^{\infty,*}_{M,x}=\sum a_{\nu,J} \; \ev_x
\frac{\partial^{\abs{\nu}}}{\partial x^\nu}$ because of the
classical theory.

Moreover it is also possible to prove that for each element $v\in \sheaf(M)^\ast$ there exists a finite number of points $x_i$ in $\red{M}$ such that  $v=\sum_{i} v_{x_i}$ with $v_{x_i}$ denoting a nonzero distribution with support at $x_i$.
\end{observation}

\begin{proposition} \label{prop:distributions}
Let $A$ be a super Weil algebra and $A^*$ its dual. Let
$    x_A \colon \stalk{M}{x} \to A$
be an $A$\nbd point near $x \in \topo{M}$ (see obs. \ref{obs:stalk}).
If $\omega \in A^*$, then
$    \omega \circ x_A \in \distr{M}{x} \text{.}$
Moreover each element of $\distr[k]{M}{x}$ can be obtained in this
way with
 $   A = \stalk{M}{x} / \maxid_x^{k+1}$ 
(see lemma \ref{lemma:SWA}).
\end{proposition}

\begin{proof}
If $A$ has height $k$, since $x_A(\maxid_x) \subseteq \nil{A}$, $\omega \circ x_A \in
\distr[k]{M}{x}$. If vice versa $v \in \distr[k]{M}{x}$, it factorizes through
$    \stalk{M}{x} \stackrel{\pr}{\to} \stalk{M}{x} / \maxid_x^{k+1} \stackrel{\omega}{\to} {\R}$
with a suitable $\omega$.
\end{proof}

In the next observation we relate the finite support distributions
and their interpretation via the Weil--Berezin functor,
to the tangent superspace.

\begin{observation} \label{obs:tangent_bundle}
Let us first recall that the tangent superspace to a smooth supermanifold
$M$ at a point $x$ is the super vector space consisting of all the
$\ev_x$\nbd derivations of  $\sheaf(M)$:
\[
    T_x(M) \coloneqq \set{ v \colon \sheaf_{M} \to {\R} | v(f\cdot g)=v(f)\ev_x(g)+\ev_x(f)v(g) } \text{.}
\]

As in the classical setting we can recover the tangent space by
using the super Weil algebra of \emph{super dual numbers} $A =
{\R}(e,\epsilon)={\R}[e,\epsilon]/ \langle e^2,
e\epsilon, \epsilon^2 \rangle$ 
(see example \ref{example:SDN}). If $x_A \in M_{A}$ is near $x$
and $s,t \in \sheaf{(M)}$, we have
$    x_A(st) = \ev_x(st) + x_e(st) e + x_\epsilon(st) \epsilon$
with $x_e,x_\epsilon \colon \sheaf{(M)} \to {\R}$. On the
other hand
\begin{align*}
    x_A(st)
    &= x_A(s) x_A(t)=
    \ev_x(s) \ev_x(t)
        + \big( x_e(s) \ev_x(t) + \ev_x(s) x_e(t) \big) e \\
        &\qquad + \big( x_\epsilon(s) \ev_x(t) + \ev_x(s) x_\epsilon(s) \big) \epsilon \text{.}
\end{align*}
Then $x_e$ (resp.\ $x_\epsilon$) is a derivation  that
is zero on odd (resp.\ even) elements and so $x_e \in T_x(M)_0$
(resp.\ $x_\epsilon \in T_x(M)_1$). The map
\[
    \begin{aligned}
T(M) \coloneqq \bigsqcup_{x \in \topo{M}} T_x(M) &\to
M_{{\R}(e,\epsilon)}, \qquad
        v_0 + v_1 &\mapsto \ev_x + v_0 e + v_1 \epsilon
    \end{aligned}
\]
(with $v_i \in T_x(M)_i$) is an isomorphism of vector bundles over
$\red{M} \isom M_{\R}$, where $\red{M}$ is the classical
manifold associated with $M$, as in {section
\ref{sec:basic_def}} (see also \cite[ch.~8]{KMS} for an exhaustive
exposition in the classical case). The reader should not confuse
$T(M)$, which is the classical bundle obtained by the union of all
the tangent superspaces at the different points of $\topo{M}$, with
$\sheaf[T]_M$, which is the super vector bundle of all the
derivations of $\sheaf_M$.
\end{observation}


We now want to give a brief account on how we can
perform differential calculus using the language of $A$\nbd
points. The essential ingredient is the super version of the
transitivity theorem.

\medskip

\begin{theorem}[Weil transitivity theorem] \label{theor:transitivity}
Let $M$ be a smooth supermanifold, $A$ a super
Weil algebra and $B_0$ a purely even Weil algebra, both real. Then
 $   {(M_A)}_{B_0} \isom M_{A \otimes B_0}$
as $(A_0 \otimes B_0)$\nbd manifolds.
\end{theorem}

\begin{proof}
Let $\sheaf_{M_A}$ and $\sheaf_{M_A}^A$ be the sheaves of smooth
maps from the classical manifold $M_A$ to ${\R}$ and $A$
respectively. Clearly $\sheaf_{M_A}^A \isom A \otimes \sheaf_{M_A}$
through the map $f \mapsto \sum_i a_i \otimes \pair{a_i^*}{f}$,
where $\set{a_i}$ is a homogeneous basis of $A$.

Consider now the map
$\tau \colon \sheaf(M) \to \sheaf(M_A)^A \isom A \otimes \sheaf{(M_A)}$,
$\tau(s)=\hat{s}$,
where, if $s \in \sheaf(M)$,
$    \hat{s} \colon y_A \mapsto y_A(s)$
for all $y_A \in M_A$.

Recalling that

\medskip

\centerline{
    ${(M_A)}_{B_0} \coloneqq \Hom_{\SAlg}(\sheaf{(M_A)},B_0) \qquad
    M_{A \otimes B_0} \coloneqq  \Hom_{\SAlg}(\sheaf{(M)},A \otimes B_0)
\text{,}$
}

\medskip

\noindent we can define a map
       $ \xi \colon {(M_A)}_{B_0} \to M_{A \otimes B_0}$,
    $\xi(X) \colon s\mapsto (\id_A \otimes X)\tau({s}) \text{.}$
This definition is well-posed since $\xi(X)$ is a superalgebra map,
as one can easily check.
Fix now a chart $(U,h)$, $h \colon U \to {\R}^{p|q}$, in $M$
and denote by $(U_A,h_A)$, $\big({(U_A)}_{B_0}, {(h_A)}_{B_0}\big)$
and $(U_{A\otimes B_0}, h_{A\otimes B_0})$ the corresponding charts
lifted to $M_A$, ${(M_A)}_{B_0}$ and $M_{A \otimes B_0}$
respectively. If $\set{e_1,\ldots,e_{p+q}}$ is a homogeneous basis
of ${\R}^{p|q}$, we have (here, according to observation
\ref{obs:coordinates}, we tacitly use the identification
${\R}^{p|q}_A \isom (A \otimes {\R}^{p|q})_0$):

\medskip

\centerline{
$        {(h_A)}_{B_0} \colon {(U_A)}_{B_0}
\to (A \otimes B_0 \otimes {\R}^{p|q})_0, \qquad
        X \mapsto
\sum_{i,j} a_i \otimes X \big( h_A^*(a_i^* \otimes e_j^*) \big) \otimes e_j$
}

\medskip

\centerline{
$        h_{A\otimes B_0} \colon U_{A \otimes B_0}
\to (A \otimes B_0 \otimes {\R}^{p|q})_0, \qquad
        Y \mapsto \sum_k Y \big( h^*(e_k^*) \big) \otimes e_k \text{.}$
}

\medskip

Then, since
$    \xi(X) \big( h^*(e_k^*) \big)
    = (\id \otimes X) \big( \widehat{\smash{h^*}(e_k^*)} \big)
    = (\id \otimes X) \big( \textstyle\sum_i a_i \otimes h_A^*(a_i^* \otimes e_k^*) \big) \text{,}$
we have
$    h_{A\otimes B_0} \circ \xi \circ {(h_A)}_{B_0}^{-1} = \id_{{(h_A)}_{B_0}({(U_A)}_{B_0})} \text{.}$
This entails in particular that $\xi$ is a local $(A_0 \otimes
B_0)$\nbd diffeomorphism. The fact that it is a global
diffeomorphism follows noticing that it is fibered over the
identity.
\end{proof}

We want to briefly explain some applications of the Weil
transitivity theorem. 

\begin{definition}
If $x_A \in M_A$, we define the space of \emph{$x_A$\nbd linear
derivations} of $M$ (\emph{$x_A$\nbd derivations} for short) as the
$A$\nbd module
\begin{align*}
    \Der_{x_A} \big( \sheaf(M),A \big)
    \coloneqq \Big\{\, & X \in \HOM \big( \sheaf(M),A \big) \,\Big|\, \forall s,t \in \sheaf(M) , \\
    & X(st)= X(s) x_A(t) + (-1)^{\p{X}\p{s}}x_A(s)X(t) \,\Big\} \text{.}
\end{align*}
\end{definition}

\begin{proposition}
The tangent superspace at $x_A$ in $M_A$ canonically identifies with
$\Der_{x_A} \big( \sheaf(M), A \big)_0$.
\end{proposition}

\begin{proof}
If $\R(e)$ is the algebra of dual number (see {example
\ref{example:SDN}}), $(M_A)_{\R(e)}$ is isomorphic, as a vector
bundle, to the tangent bundle $T(M_A)$, as we have seen in
observation \ref{obs:tangent_bundle}. Due to theorem
\ref{theor:transitivity}, we thus have an isomorphism
\[
    \xi \colon T(M_A) \isom {(M_A)}_{\R(e)} \to M_{A\otimes\R(e)} \text{.}
\]
On the other hand, it is easy to see that $x_{A\otimes \R(e)} \in
M_{A\otimes\R(e)}$ can be written as $x_{A\otimes \R(e)} = x_A
\otimes 1 + v_{x_A} \otimes e$, where $x_A \in M_A$ and $v_{x_A}
\colon \sheaf(M) \to A$ is a parity preserving map satisfying the
following rule for all $s,t \in \sheaf(M)$:
\[
    v_{x_A}(st) = v_{x_A}(s) \, x_A(t) + x_A(s) \, v_{x_A}(t) \text{.}
\]
Then each tangent vector on $M_A$ at $x_A$ canonically identifies a
even $x_A$\nbd derivation and, vice versa, each such derivation
canonically identifies a tangent vector at $x_A$.
\end{proof}

We conclude studying more closely the structure of $\Der_{x_A} \big(
\sheaf(M), A \big)$. The following proposition describes it
explicitly.

Let $K$ be a right $A$\nbd module and let $L$ be a left $B$\nbd
module for some algebras $A$ and $B$. Suppose moreover that an
algebra morphism $\rho \colon B \to A$ is given. One defines the
$\rho$\nbd tensor product $K \otimes_\rho L$ as the quotient of the
vector space $K \otimes L$ with respect to the equivalence relation
$    k \otimes b \cdot l \sim k \cdot \rho(b) \otimes l$,
for all $k \in K$, $l \in L$ and $b \in B$.

Moreover, if $M$ is a supermanifold, we denote by $\sheaf[T]_M$ the
\emph{super tangent bundle} of $M$, i.~e.\ the sheaf defined by
$\sheaf[T]_M \coloneqq \Der(\sheaf_M)$.

\begin{proposition}
Let $M$ be a smooth supermanifold and let $x \in \topo{M}$. Denote
$\stalk[T]{M}{x}$ the germs of vector fields at $x$. One has the
identification of left $A$\nbd modules
\[
    \Der_{x_A} \big( \sheaf(M), A \big)
    \isom A \otimes T_{\red{x_A}}(M)
    \isom A \otimes_{x_A} \stalk[T]{M}{\red{x_A}} \text{.}
\]
\end{proposition}

This result is clearly local so that it is enough to prove it in the
case $M$ is a superdomain. Next lemma does this for the first
identification. The second descends from eq.\
\eqref{eq:x_A-derivation}, since $\stalk[T]{M}{\red{x_A}} =
\stalk{M}{\red{x_A}} \otimes T_{\red{x_A}}(M)$, where $\stalk{M}{\red{x_A}}$ denotes the stalk
at ${\red{x_A}}$.

\begin{lemma}
Let $U$ be a superdomain in $\R^{p|q}$ with coordinate system
$\set{x_i,\theta_j}$, $A$ a super Weil algebra, and $x_A \in U_A$.
To any list of elements
\[
    \f = (f_1,\ldots,f_p,F_1,\ldots,F_q)
    \qquad f_i, F_j \in A
\]
there corresponds a $x_A$\nbd derivation
$    X_{\f} \colon \sheaf(U) \to A$
given by
\begin{equation} \label{eq:x_A-derivation}
    X_{\f}(s) = \sum_i f_i \, x_A \left( \pd{s}{x_i} \right) + \sum_j F_j \, x_A \left( \pd{s}{\theta_j} \right) \text{.}
\end{equation}
$X_{\f}$ is even (resp.\ odd) if and only if the $f_i$ are even
(resp.\ odd) and the $F_j$ are odd (resp.\ even). Moreover any
$x_A$\nbd derivation is of this form for a uniquely determined $\f$.
\end{lemma}

\begin{proof}
That $X_{\f}$ is a $x_A$\nbd derivation is clear. That the family
$\f$ is uniquely determined is also immediate from the fact that
they are the value of $X_{\f}$ on the coordinate functions.

Let now $X$ be a generic $x_A$\nbd derivation. Define
$    f_i = X(x_i) \text{,}$
$    F_j = X(\theta_j) \text{,}$
and
 $   X_{\f} = f_i \, x_A \circ \pd{}{x_i} + F_j \, x_A \circ \pd{}{\theta_j} \text{.}$
Let $D = X-X_{\f}$. Clearly $D(x_i) = D(\theta_j) = 0$. We now show
that this implies $D=0$. Let $s\in\sheaf(U)$. Due to lemma
\ref{lemma:polynomials}, for each $x \in U$ and for each integer
$k\in\N$ there exists a polynomial $P$ in the coordinates such that
${s} - {P} \in \maxid_{U,x}^{k+1}$. Due to Leibniz rule
$D(s - P) \in \nil{A}^k$ and, since clearly $D(P)= 0$, $D(s)$ is in
$\nil{A}^k$ for arbitrary $k$. So we are done.
\end{proof}


\begin{corollary}
We have: 
$    T_{x_A} M_A
    \isom \big( A \otimes T_{\red{x_A}}(M) \big)_0
    \isom \big( A \otimes_{x_A} \stalk[T]{M}{\red{x_A}} \big)_0 \text{.}$
\end{corollary}


\end{document}